\documentclass[12pt]{article}
\usepackage{geometry}
\usepackage{amsfonts}
\usepackage{amsmath}
\usepackage{amsthm}
\usepackage{amssymb}
\usepackage{mathrsfs}
\usepackage{graphicx}

\usepackage[varg]{txfonts}

\numberwithin{equation}{section}

\theoremstyle{plain}
\newtheorem{theorem}{Theorem}[section]
\newtheorem{lemma}[theorem]{Lemma}

\newtheorem{proposition}[theorem]{Proposition}

\theoremstyle{definition}
\newtheorem{definition}[theorem]{Definition}
\newtheorem*{hypothesis}{Standing assumption}

\theoremstyle{remark}
\newtheorem{remark}[theorem]{Remark}
\newtheorem{example}[theorem]{Example}

\newcommand{\md}{\mathrm{d}}
\newcommand{\me}{\mathrm{e}}
\newcommand{\R}{\mathbb{R}}
\newcommand{\E}{\mathbb{E}}

\newcommand{\vf}{\varphi}

\newcommand{\minev}{\lambda_{*}}

\newcommand{\svbset}{\mathfrak{S}}
\newcommand{\Sym}{\mathcal{S}}
\newcommand{\Ctr}{\Upsilon}

\begin{document}

\title{\bf On solvability of an indefinite Riccati equation}

\author{Kai~Du\thanks{Financial support by the National Centre of Competence
in Research ``Financial Valuation and Risk Management'' (NCCR FINRISK),
Project D1 (Mathematical Methods in Financial Risk Management) is gratefully
acknowledged. The NCCR FINRISK is a research instrument of the Swiss National
Science Foundation.}\medskip\\ 
Department of Mathematics, ETH Zurich\\
{\tt kai.du@math.ethz.ch}}

\maketitle



\begin{abstract}
This note concerns a class of matrix Riccati equations associated with stochastic linear-quadratic optimal control problems with indefinite state and control weighting costs.
A novel sufficient condition of solvability of such equations is derived,
based on a monotonicity property of a newly defined set. 
Such a set is used to describe a family of solvable equations.
\bigskip

\noindent {\bf Key words.} indefinite Riccati equation, LQ problem, solvable set, solvability condition, quasi-linearization
\bigskip

\noindent {\bf MSC 2010 Subject Classiﬁcation.}
34A12, 49N10, 93E20
\end{abstract}

\section{Introduction}

In this paper, we study the problem of solvability of the following matrix Riccati differential equation over a running time interval $[0,T]$:
\begin{subequations}
\begin{equation}\label{1-001}
\begin{aligned}
& \dot{P} + A^\top P + P A + C^\top P C + Q \\
& \quad  = (PB + C^\top P D) (R + D^\top P D)^{-1}(PB + C^\top P D)^{\top}
\end{aligned}
\end{equation}
subject to the terminal condition $P(T) = G$, and the constraint that
\begin{equation}\label{1-002}
R + D^\top P D > 0
\quad \text{over } [0,T],
\end{equation}
\label{1-000}\end{subequations}
where $P$ is the unknown matrix-valued function, and $A,B,C,D,R,Q$ and $G$ are given data.
The time parameter $t$ is omitted for simplicity in this formulation.
More specifically, the following assumption is made throughout this paper.
\begin{hypothesis}
The data appearing in \eqref{1-000} satisfy that
\begin{align*}
A,B,C \in L^{\infty}(0,T;\R^{d \times d}),& \quad
Q \in L^{\infty}(0,T;\Sym^d),\quad
G \in \Sym^d,\\
D \in C([0,T];\R^{d \times d}),&\quad
R \in C([0,T];\Sym^d),
\end{align*}
where $\Sym^d$ is the set of symmetric $d\times d$ matrices.
\end{hypothesis}
In view of the known uniqueness of solutions, a solution matrix $P$ must take values in $\Sym^d$.
As in \eqref{1-002}, we use inequality signs to express the usual semi-order of $\Sym^d$ throughout this paper.

The solvability of \eqref{1-000} plays a key role when solving the following stochastic linear-quadratic (LQ) optimal control problem:
seeking a control $u(\cdot)$ that minimizes the cost
\begin{subequations}
\begin{equation}\label{1-004a}
J(u) = \E\,\biggl\{\int_0^T \left[ u(t)^\top R(t) u(t) + x(t)^\top Q(t) x(t)\right]\md t + x(T)^\top G x(T)\biggr\}
\end{equation}
subject to the controlled state system
\begin{equation}\label{1-004b}
\md x = (Ax + Bu)\,\md t + (Cx + Du)\,\md w_t,\quad
x(0) = x_0 \in \R^d,
\end{equation}
\label{1-004}\end{subequations}
where $(w_t)_{t\ge 0}$ is a standard one-dimensional Wiener process (cf.~\cite{chen1998stochastic, young1999stochastic}). More specifically, as long as \eqref{1-000} admits a solution $P$ over the whole time interval, the above LQ problem is solvable with the minimal cost $J^* = x_0^\top P(0) x_0$ and an optimal control
\[
u^*(t) = - (R + D^\top P D)^{-1}(PB + C^\top P D)^{\top} x(t).
\]
Due to this known connection, we call $(A,B,C,D)$ the \emph{system data}, and call $(R,Q,G)$ the \emph{cost weights}.

In the existing literature $Q$ and $G$ were usually assumed to be semi-positive.
In the early formulation the constraint~\eqref{1-002} did not involved, since the assumption that $R>0$, inherited from deterministic LQ problems, was taken for granted that evidently implies~\eqref{1-002}.
This case has been fully solved via different methods (cf.~\cite{bismut1976linear, peng1992stochastic, tang2003general}).
The first formulation combining \eqref{1-001} and \eqref{1-002} was introduced by Chen-Li-Zhou~\cite{chen1998stochastic}, under both theoretical and practical considerations, in the study of indefinite LQ problems where $R$ could be indefinite.
The uniqueness of solutions to \eqref{1-000} has been established in great generality, see~\cite[Proposition 6.7.2]{young1999stochastic} for instance.
But the existence part is much more complicated.
It is worth noting that the existence is by no means unconditional (see~\cite{chen1998stochastic} for ill-posed examples); 
the problem is thus to find sufficient conditions that ensure the existence.
It has been solved only for several very special cases so far. 
For instance, the one-dimensional problem with constant coefficients has been completely resolved in~\cite{young1999stochastic}.
A necessary-sufficient condition of solvability was given in~\cite{chen1998stochastic} for the case $C=0$, but rather implicit.
Chen-Zhou~\cite{chen2000stochastic} settled the case that $R=0$ and $D^\top D > 0$ allowing different state and control dimensions.
Apart from \eqref{1-000}, its stochastic version (i.e., indefinite stochastic Riccati equations) is also considered in the existing literature, see, e.g.~\cite{hu2003indefinite,qian2013existence}.
However, the deterministic implications of those results either have been derived in some earlier publications or require certain strong conditions such as an equality constraint on the coefficients.
To our best knowledge, the general existence remains to this date an open problem.

In this paper we derive a novel class of sufficient conditions of solvability of \eqref{1-000} thanks to a new insight into the problem.
In view of \eqref{1-004},
we prefer to characterize the set $\svbset$ of cost weights $(R,Q,G)$ of all solvable \eqref{1-000} for fixed system data.
The solvability of the equation is then implied in a complete characterization of $\svbset$.
The latter seems no less difficult than the former, but stimulates a novel approach to derive a class of sufficient conditions. 
This is due to a basic observation, elaborated in Theorem~\ref{thm:201} below, that $\svbset$ is convex, and especially, monotonic in the sense that if $(R,Q,G) \in \svbset$ and $R\le \bar{R},Q \le \bar{Q},G\le\bar{G}$, then $(\bar{R},\bar{Q},\bar{G})\in\svbset$;
in other words, each indefinite triple $(R,Q,G)$ in $\svbset$ can be regarded as a ``benchmark'' that yields a family of solvable indefinite Riccati equations of form~\eqref{1-000}, and, in a sense, the ``lower'' the better.
As a first attempt of this new idea, we derive a sufficient condition of solvability of \eqref{1-000} stated in the following theorem, where we denote by $\minev(M)$ the minimal eigenvalue of a symmetric matrix $M$.

\begin{theorem}\label{thm1.1}
Let $D$ be invertible, and $G>0$. 
For $\alpha\in(0,1)$, let
\[
\lambda_\alpha(t) = 
\minev \bigl(\bigl[{A}^\top + {A} + C^\top C
- (1-\alpha)^{-1}(B+C^\top D) (D^\top D)^{-1} (B^\top + D^\top C) \bigr](t)\bigr),
\]
and $\vf_\alpha(\cdot)>0$ satisfy the following linear ODE over $[0,T]$:
\begin{equation}\label{1-003}
\dot{\vf}_{\alpha}(t) + \lambda_\alpha(t) \vf_{\alpha}(t) + \alpha \minev(Q(t)) = 0,
\quad {\vf}_{\alpha}(T) = \alpha \minev(G).
\end{equation}
Then \eqref{1-000} admits a unique solution over $[0,T]$ provided $R \ge -\vf_{\alpha}D^\top D$ for some $\alpha$ over the same time interval.
\end{theorem}

In the above result we determine a class of ``benchmarks'' via a family of linear ODEs,
which is quite explicit and simple.
Such a condition, that is of course not optimal (see Example~\ref{ex5-04} below), is basically either more general or more practicable than those in the existing literature.
From it one can easily construct various examples of solvable indefinite Riccati equations and associated LQ problems for given system data, more interestingly, including those in which not only $R$ but also $Q$ is indefinite.
Likely some refinement of the analysis will yield a more precise characterization of $\svbset$. 
This is planned as future work and is beyond the scope of this note.

The rest of the paper is organized as follows.
Section~2 is devoted to the statement of main results and the proof of Theorem~\ref{thm1.1}.
In Section~3 we utilize Bellman's quasi-linearization method to prove Theorem~\ref{thm:201} that gives some fundamental properties of $\svbset$.
Finally, Section~4 consists of several concrete examples.

\section{Main results}

The main idea of our approach is to characterize cost weights $(R,Q,G)$ that ensure solvability of \eqref{1-000}  for fixed system data $(A,B,C,D)$.
To this end, we introduce the following definition.

\begin{definition}
The \emph{solvable set} $\svbset$ associated with the coefficients $(A,B,C,D)$ is the set of all triples $(R,Q,G)$ such that the corresponding Riccati equation \eqref{1-000} admits a solution.
\end{definition}

The first main result as follows is several fundamental properties of solvable sets.

\begin{theorem}\label{thm:201}
$\svbset$ is a nonempty convex cone excluding the origin. 
More specifically,
denoting by $P(R,Q,G)$ the solution of \eqref{1-000} associated with $(R,Q,G)\in \svbset$, we have the following properties:

{\rm (a)} Positive homogeneity: 
if $(R,Q,G)\in \svbset$ and $\lambda>0$, then $(\lambda R,\lambda Q,\lambda G) \in \svbset$ and $P(\lambda R,\lambda Q,\lambda G) = \lambda P(R,Q,G)$.

{\rm (b)} Super-additivity: 
if $(R,Q,G)\in \svbset$ and $(\bar{R},\bar{Q},\bar{G})\in \svbset$, then $(R + \bar{R},Q+\bar{Q},G+\bar{G})\in \svbset$, and
\begin{align*}
P(R + \bar{R}, Q + \bar{Q}, G + \bar{G}) \ge P(R,Q,G) + P(\bar{R},\bar{Q},\bar{G}).
\end{align*}

{\rm (c)} Monotonicity: 
if $(R,Q,G)\in \svbset$, and $R\le \bar{R}$, $Q\le\bar{Q}$, $G\le \bar{G}$, then $(\bar{R},\bar{Q},\bar{G})\in \svbset$ and $P(R,Q,G) \le P(\bar{R},\bar{Q},\bar{G})$.
\end{theorem}

The above result is intuitively natural from the viewpoint of control,
but technically nontrivial.
Indeed, the solvability of the Riccati equation yields the well-posedness of the associated LQ problem but not usually vice versa.

Let us first show that $(0,0,0)\notin\svbset$.
Indeed, if it is not true, then the associated LQ problem \eqref{1-004} is solvable for any initial time $s\in[0,1]$ and data $x_0$, with the minimal cost $J^* \equiv 0$.
This implies $P = 0$ that does not satisfy \eqref{1-002}.

The assertion (a) of Theorem~\ref{thm:201} follows directly from scaling.
The solvable set $\svbset$ is nonempty due to the following known result (cf.~\cite[Theorem 6.7.2]{young1999stochastic}).

\begin{lemma}\label{lem:201}
Let $R>0$, $G\ge 0$ and $Q-S R^{-1}S^\top \ge 0$ with $S \in L^{\infty}(0,T;\R^{d\times d})$.
Then the Riccati equation
\begin{equation*}
\begin{aligned}
& \dot{P} + A^\top P + P A + C^\top P C + Q \\
& \quad  = (PB + C^\top P D + S^\top) (R + D^\top P D)^{-1}(PB + C^\top P D + S)^{\top}
\end{aligned}
\end{equation*}
with $P(T) = G$ admits a unique solution $P\in C([0,T];\Sym^d)$ with $P\ge 0$.
\end{lemma}

Let us postpone the rest of the proof of Theorem~\ref{thm:201} into next section, but turn to prove another main result, Theorem~\ref{thm1.1}, that has been stated in the previous section.
The existence part follows immediately from the following result that is more general but less practicable than Theorem~\ref{thm1.1}.

\begin{proposition}\label{prp:301}
Let $D=I$, $G>0$, 
$\alpha \in (0,1)$, and $R_\alpha(\cdot) > 0$ satisfying $R_{\alpha}(T) \le \alpha G$ and
\begin{align}\label{3-001}
\nonumber & \dot{R}_{\alpha} + A^\top R_{\alpha} + R_{\alpha} A + C^\top R_{\alpha} C + \alpha Q\\
& \quad - (1-\alpha)^{-1}(R_{\alpha}B + C^\top R_{\alpha})R_{\alpha}^{-1}(R_{\alpha}B + C^\top R_{\alpha})^\top \ge 0
\end{align}
over $[0,T]$.
Then $(R,Q,G)\in \svbset$ provided $R + R_{\alpha} \ge 0$.
\end{proposition}

\begin{proof}
Fix an $\alpha \in (0,1)$.
Consider the Riccati equation
\begin{equation}\label{3-002}
\left\{\begin{aligned}
& \dot{K} + A^\top K + K A + C^\top K C + \tilde{Q} \\
& \quad = (K B +  C^\top K + \tilde{S}^\top) ( \tilde{R}_{\alpha} + K)^{-1}(K B +  C^\top K + \tilde{S})^{\top},\\
& K_T = G - \alpha^{-1} R_{\alpha} \ge 0
\end{aligned}\right.
\end{equation}
where
\begin{align*}
\tilde{Q} & = {\alpha}^{-1}(\dot{R}_{\alpha} + A^\top R_{\alpha} + R_{\alpha} A + C^\top R_{\alpha} C) + Q,\\
\tilde{S} & = {\alpha}^{-1}(B^\top R_{\alpha} +  R_{\alpha} C),
\quad
\tilde{R}_\alpha = {\alpha}^{-1}(1-\alpha) R_{\alpha}.
\end{align*}
The condition~\eqref{3-001} implies that
\[
\tilde{Q} - \tilde{S}  \tilde{R}_{\alpha}^{-1} \tilde{S}^\top \ge 0.
\]
In view of Lemma~\ref{lem:201}, \eqref{3-002} admits a unique solution $K(\cdot) \ge 0$ over $[0,T]$.
Set $P = K + \alpha^{-1}R_\alpha$.
From~\eqref{3-002} it is easy to verify that $P(T) = G$, and $P-R_\alpha = K + \tilde{R}_\alpha > 0$, and
\begin{align*}
& \dot{P} + A^\top P + P A + C^\top P C + Q \\
& = \dot{K} + A^\top K + K A + C^\top K C + {\alpha}^{-1}(\dot{R}_{\alpha} + A^\top R_{\alpha} + R_{\alpha} A + C^\top R_{\alpha} C) + Q\\
& = (K B +  C^\top K + \tilde{S}^\top) ( \tilde{R}_{\alpha} + K)^{-1}(K B +  C^\top K + \tilde{S})^{\top}\\
& = (PB + C^\top P) (P - R_\alpha)^{-1}(PB + C^\top P)^{\top},
\end{align*}
Thus $(-R_\alpha,Q,G)\in\svbset$, which along with Theorem~\ref{thm:201} (c) concludes the proof.
\end{proof}

\begin{remark}\label{rmk201}
In the above proof we have seen that $P = K + \alpha^{-1}R_\alpha
\ge \alpha^{-1}R_\alpha$. 
This by means of Theorem~\ref{thm:201}(c) provides a low bound of the solution to \eqref{1-000}.
\end{remark}

\begin{proof}[Proof of Theorem~\ref{thm1.1}]
Without loss of generality, we assume $D=I$ in addition.
Then the existence follows immediately from Proposition~\ref{prp:301} by taking $R_\alpha(t) = \vf(t)I$ and verifying \eqref{3-001} from \eqref{1-003}.
The uniqueness is a well-know result, for instance, see~\cite[Proposition~6.7.2]{young1999stochastic}. 
The proof is complete.
\end{proof}

\begin{remark}
Let us point out that the result of Theorem \ref{thm:201} can be extended trivially to a more general case that the Riccati equation reads
\[\left\{
\begin{aligned}
& \dot{P} + A^\top P + P A + \textstyle{\sum}_{i=1}^{m} C_i^\top P C_i + Q \\
& \quad  = \big(PB + \textstyle{\sum}_{i=1}^{m} C_i^\top P D_i\big) \big(R + \textstyle{\sum}_{i=1}^{m} D_i^\top P D_i\big)^{-1}\big(PB + \textstyle{\sum}_{i=1}^{m} C_i^\top P D_i\big)^{\top},\\
& R + \textstyle{\sum}_{i=1}^{m} D_i^\top P D_i > 0, \quad \text{over the time interval }[0,T],\text{ and}\\
& P(T) = G,
\end{aligned}\right.
\]
and with different state and control dimensions, i.e., $B$ and $D$ take values in $\R^{d\times k}$, and $R$ in $\Sym^k$,
but in this note Theorem \ref{1-000} could be derived only in the case that $k=d$ and $m=1$.
For a balanced consideration, we shall restrict ourselves to \eqref{1-000} rather than a general form as above.
\end{remark}

\section{Quasi-linearization}

In this section we shall complete the proof of Theorem~\ref{thm:201} by virtue of Bellman's quasi-linearization method that has been used to resolve the standard Riccati equation, i.e., $R>0$ and $Q,G\ge 0$ (see~\cite{wonham1968matrix, young1999stochastic} for instance).
First of all we state a preliminary lemma taken from~\cite[Lemma~6.7.3]{young1999stochastic}.

\begin{lemma}\label{lem401}
The following linear matrix ODE over time interval $[0,T]$:
\[
\dot{P} + A^\top P + P A + C^\top P C + Q = 0,
\quad P(T) = G
\]
admits a unique solution $P\in C([0,T];\Sym^d)$. 
If $Q\ge 0$ and $G \ge 0$, then $P \ge 0$.
\end{lemma}

The first step of the quasi-linearization method is typically to rewrite \eqref{1-001} into the following form:
\begin{equation}\label{4-003}
\dot{P} + \Phi(P,\Ctr(P)) + Q + \Ctr(P)^\top R \Ctr(P) = 0,
\end{equation}
with $P(T) = G$, where
\begin{align*}
\Phi(P,U) & := (A+BU)^\top P + P (A+BU) + (C+DU)^\top P (C+DU)\\
\Ctr(P) & := - (R + D^\top P D)^{-1}(B^{\top} P + D^\top P C)
\end{align*}
A direct calculation shows that
\begin{equation}\label{4-001}
\begin{aligned}
\Phi(P,U) + U^\top R U & - \Phi(P,\Ctr(P)) - \Ctr(P)^\top R \Ctr(P)\\
& = (U-\Ctr(P))^{\top} (R + D^\top P D) (U-\Ctr(P)).
\end{aligned}
\end{equation}
This equality plays an important role in our arguments.

Although it seems hopeless to apply Bellman's quasi-linearization method to resolve indefinite Riccati equations completely,
the following result tells us that such a approach can still be used to obtain a necessary-sufficient condition of solvability of \eqref{1-000}.

\begin{proposition}\label{prp401}
Define a sequence $(P_n(\cdot))_{n\ge 0}$ recursively as
\begin{equation}\label{4-002}
\begin{aligned}
& P_0 = 0; \quad P_n(T) = G,\\
& \dot{P}_{n} + \Phi(P_{n},\Ctr(P_{n-1})) + Q + \Ctr(P_{n-1})^\top R \Ctr(P_{n-1}) = 0,
\quad n\ge 1.
\end{aligned}
\end{equation}
Then \eqref{1-000} admits a solution over $[0,T]$ if and only if
\eqref{4-002} admits a solution for each $n\ge 1$, and moreover, there are constants $c,\delta >0$ such that
\begin{equation}\label{4-006}
P_n \ge -c I
\quad\text{and}\quad
R + D^\top P_n D \ge \delta I \quad\text{over }[0,T] \text{ for each }
n\ge 1.
\end{equation}
\end{proposition}

\begin{proof}
\emph{Necessity}.
Let us prove this direction by induction.
Suppose the assertion holds true for $n-1$. 
Then the existence and uniqueness of $P_n$ follow from Lemma~\ref{lem401}.
In view of \eqref{4-003} and \eqref{4-001}, we have
\begin{align*}
\dot{P} = & - \Phi(P,\Ctr(P)) - \Ctr(P)^\top R \Ctr(P) - Q\\
= & - \Phi(P,\Ctr(P_{n-1})) - \Ctr(P_{n-1})^\top R \Ctr({P}_{n-1}) - Q + \Theta_n.
\end{align*}
where $\Theta_n := (\Ctr(P_{n-1})-\Ctr(P))^{\top} (R + D^\top P D) (\Ctr(P_{n-1})-\Ctr(P)) \ge 0$.
Combining the equation of $P_n$, we have that
\[
\md (P_n - P)/\md t
+ \Phi(P_n - P, \Ctr(P_{n-1})) + \Theta_n = 0,
\quad (P_n - P)(T)=0.
\]
Thanks to Lemma~\ref{lem401}, we know that $P_n \ge P$, and then
from \eqref{1-002}, $R + D^\top P_n D \ge R + D^\top P D >0$, thus
\eqref{4-006} holds for $n$.

Since $P_1$ does exist from Lemma~\ref{lem401},
one can analogously prove the assertion for $n=1$. Therefore, by induction we conclude the necessity.

\emph{Sufficiency}.
An analogous calculation yields
\[
\md (P_n - P_{n+1})/\md t
+ \Phi(P_n - P_{n+1}, \Ctr(P_{n})) + \Delta_n = 0,
\quad (P_n - P_{n+1})(T)=0,
\]
where 
$$\Delta_n := (\Ctr(P_{n-1})-\Ctr(P_n))^{\top} (R + D^\top P_n D) (\Ctr(P_{n-1})-\Ctr(P_n)) \ge 0.$$
From Lemma~\ref{lem401} and the fact that $R + D^\top P_n D \ge \delta I$ for all $n$, we know $P_n - P_{n+1} \ge 0$ for any $n\ge 1$.
Thus $(P_n)_{n\ge 1}$ is a decreasing sequence in $C([0,T];\Sym^d)$, which along with the fact that $P_n \ge -cI$ for all $n$ yields that $(P_n)_{n\ge 1}$ has a limit, denoted by $P$. 
Clearly, $P$ is the solution of \eqref{1-000}.
The proof is complete.
\end{proof}

\begin{remark}
Equations \eqref{4-002} actually constitute a numerical algorithm to compute the solution of \eqref{1-000}.
By an analogous argument as that in~\cite[Proposition~4.1]{chen2000stochastic}, one can derive an estimate for the convergence speed of this algorithm as follows:
\[
|P_n(t) - P(t)| \le
K \sum_{k=n-2}^{\infty}
\frac{M^k}{k!}(T-t)^{k},\quad 
n \ge 3,
\]
where $K,M > 0$ are constants independent of $n$ and $t$.
\end{remark}

We are now in a position to complete the proof of Theorem~\ref{thm:201}.

\begin{proof}[Proof of Theorem~\ref{thm:201}]
It is clear that the convexity of $\svbset$ follows from properties (a) and (b).
Now let us prove assertion (b) by induction. Write
\begin{equation}\label{4-005}
(\tilde{R},\tilde{Q},\tilde{G}) = (R + \bar{R},Q+\bar{Q},G+\bar{G}).
\end{equation}
In order to apply Proposition~\ref{prp401}, we define (formally) sequence
$(\tilde{P}_n(\cdot))_{n\ge 0}$ recursively as
\begin{equation}\label{4-010}
\begin{aligned}
& \tilde{P}_0 = 0; \quad \tilde{P}_n(T) = \tilde{G},\\
& \dot{\tilde{P}}_{n} + \Phi(\tilde{P}_{n},\Ctr(\tilde{P}_{n-1})) + \tilde{Q} + \Ctr(\tilde{P}_{n-1})^\top \tilde{R} \Ctr(\tilde{P}_{n-1}) = 0,
\quad n\ge 1.
\end{aligned}
\end{equation}
Suppose \eqref{4-010} are well-defined up to $n-1$, and 
\begin{align*}
\tilde{P}_{n-1} \ge P + \bar{P}
\quad\text{and}\quad
\tilde{R} +  D^\top \tilde{P}_{n-1} D
\ge R + D^\top P D + \bar{R} + D^\top \bar{P} D > 0.
\end{align*}
Then $\tilde{P}_n$ is also well-defined from Lemma~\ref{lem401}. 
It follows from \eqref{4-003} and \eqref{4-001} that
\begin{align}\label{4-011}
\nonumber \dot{P} = & - \Phi(P,\Ctr(\tilde{P}_{n-1})) - \Ctr(\tilde{P}_{n-1})^\top R \Ctr(\tilde{P}_{n-1}) - Q\\
& + (\Ctr(\tilde{P}_{n-1})-\Ctr(P))^{\top} (R + D^\top P D) (\Ctr(\tilde{P}_{n-1})-\Ctr(P)).
\end{align}
Do the same transformation with respect to $\bar{P}$. 
Then, with \eqref{4-005} in mind we obtain
\begin{equation}\label{4-021}
\bigg\{\begin{aligned}
& \md (\tilde{P}_{n} - P - \bar{P})/{\md t}
+ \Phi(\tilde{P}_{n} - P - \bar{P}, \Ctr(\tilde{P}_{n-1}))
+ \Theta_n = 0, \\
& (\tilde{P}_{n} - P - \bar{P})(T) = 0,
\end{aligned}
\end{equation}
where
\begin{align*}
\Theta_n \, := \ & (\Ctr(\tilde{P}_{n-1})-\Ctr(P))^{\top} (R + D^\top P D) (\Ctr(\tilde{P}_{n-1})-\Ctr(P))\\
& + (\Ctr(\tilde{P}_{n-1})-\Ctr(\bar{P}))^{\top} (\bar{R} + D^\top \bar{P} D) (\Ctr(\tilde{P}_{n-1})-\Ctr(\bar{P})) \ge 0.
\end{align*}
By means of Lemma~\ref{lem401}, one has that
\begin{align}\label{4-020}
\tilde{P}_{n} \ge P + \bar{P}
\quad\text{and}\quad
\tilde{R} +  D^\top \tilde{P}_{n} D
\ge R + D^\top P D + \bar{R} + D^\top \bar{P} D > 0.
\end{align}
Since $\tilde{P}_1$ is well-defined due to Lemma~\ref{lem401}, and can be proved by a similar argument to satisfy \eqref{4-020} with $n=1$.
Thus, by induction, \eqref{4-020} holds for each $n\ge 1$.
This along with Proposition~\ref{prp401} concludes assertion (b) of Theorem~\ref{thm:201}.

Analogously, we can prove the assertion (c).
Indeed, one just needs to repeat the above argument with $(\bar{P},\bar{R},\bar{Q},\bar{G})$ instead of $(\tilde{P},\tilde{R},\tilde{Q},\tilde{G})$ in \eqref{4-010} and \eqref{4-011}, with \eqref{4-021} instead of the following
\[
\bigg\{\begin{aligned}
& \md (\bar{P}_n - P)/{\md t}
+ \Phi(\bar{P}_n - P, \Ctr(\bar{P}_{n-1}))
+ \varTheta_n = 0, \\
& (\bar{P}_n - P)(T) = \bar{G} - G \ge 0
\end{aligned}
\]
where (recalling that $Q \le \bar{Q}$ and $R\le \bar{R}$)
\[
\varTheta_n := \bar{Q} - Q + (\Ctr(\bar{P}_{n-1})-\Ctr(P))^{\top} (\bar{R} - R + D^\top P D) (\Ctr(\bar{P}_{n-1})-\Ctr(P)) 
\ge 0.
\]
Thus, also by induction, one can show that for each $n \ge 1$,
\begin{align*}
\bar{P}_{n} \ge P
\quad\text{and}\quad
\bar{R} +  D^\top \bar{P}_{n} D
\ge R + D^\top P D > 0.
\end{align*}
This by Proposition~\ref{prp401} concludes the assertion (c). 
The proof is complete.
\end{proof}

\section{Examples}

This section is devoted to several examples that illustrate the main results of this paper, especially Theorem~\ref{thm1.1}.
All of them will appear in the form of control problems
due to a well-known connection.

\begin{example}\label{ex5-01}
Consider the two-dimensional control problem: minimizing
\[
J = \E \int_0^1 \left( r_1(t)|u_1(t)|^2 + r_2(t)|u_2(t)|^2 \right)\md t + \E \left( |x_1(1)|^2 + |x_2(1)|^2\right),
\]
subject to
\[
\bigg\{
\begin{aligned}
& \md x_1(t) = \left(a_1 x_1(t) + u_2(t)\right)\md t + \left(x_2(t) + u_1(t)\right)\md w_t, \quad x_1(0) = y_1; \\
& \md x_2(t) = \left(a_2 x_2(t) - u_1(t)\right)\md t + \left(x_1(t) + u_2(t)\right)\md w_t, \quad x_2(0) = y_2,
\end{aligned}
\]
where $a_1$ and $a_2$ are given constants. In this case, the system data are
\begin{gather*}
A = \begin{bmatrix}
a_1 &  \\
 & a_2
\end{bmatrix}
,\quad
B = \begin{bmatrix}
 & 1 \\
-1 & 
\end{bmatrix}
,\quad
C = \begin{bmatrix}
 & 1 \\
1 & 
\end{bmatrix}
,\quad
D=I.
\end{gather*}

Let $\alpha \in (0,1)$. A straightforward computation shows that
\[
{A}^\top + {A} + C^\top C
- \frac{1}{1-\alpha}(B+C^\top) (B^\top+C)
=
\begin{bmatrix}
2a_1 + 1 - \frac{1}{1-\alpha} & \\
& 2a_2 + 1
\end{bmatrix}.
\]
In view of Theorem~\ref{thm1.1}, taking
\[\lambda_\alpha = 1 + \min\Big\{2a_1 - \frac{1}{1-\alpha}, 2a_2\Big\}
\quad\text{and}\quad
{\vf}_{\alpha}(t) = \alpha \me^{\lambda_\alpha (1-t)},
\]
we obtain that this LQ problem and the associated Riccati equation are both solvable as long as, for some $\alpha \in (0,1)$,
\[
\min\{r_1(t), r_2(t)\} \ge  - \alpha \me^{\lambda_\alpha (1-t)},\quad \forall\,t\in[0,1].
\]
This tells us that both $r_1(t)$ and $r_2(t)$ can be negative.
Furthermore, in view of Remark~\ref{rmk201}, we can give a lower bound of the solution of the associated Riccati equation that 
\[ P(t) \ge \alpha^{-1}{\vf}_{\alpha}(t) I = \me^{\lambda_\alpha (1-t)} I,
\quad
\forall\,t\in[0,1]. 
\]
\end{example}

\begin{example}\label{ex5-02}
Consider the one-dimensional control problem: minimizing
\[
J = \E \int_0^1 \left[ r(t)|u(t)|^2 + q(t)|x(t)|^2 \right]\md t + \E \left[ |x(1)|^2 \right],
\]
subject to
\[
\md x(t) = \left(a x(t) + b u(t)\right)\md t + \left( c x(t) + u(t)\right)\md w_t, \quad x(0) = x_0,
\]
where $a,b$, and $c$ are given constants.

To apply Theorem~\ref{thm1.1}, for $\alpha \in (0,1)$,
take
$\lambda_\alpha = 2a + c^2 - \frac{1}{1-\alpha}(b+c)^2$, and
\begin{equation*}
{\vf}_{\alpha}(t) = \alpha \me^{\lambda_\alpha (1-t)}
\left(1 + \int_t^1 \me^{\lambda_\alpha (s-1)} q(s) \,\md s\right).
\end{equation*}
Then, as long as ${\vf}_{\alpha} > 0$ the optimal control problem is solvable with control weight $r(t)\ge - {\vf}_{\alpha}(t)$.
To ensure ${\vf}_{\alpha} > 0$, we can select a state weight $q(\cdot)$ such that, for instance,
\begin{equation*}
q(t) > - 1 \ \ \text{if } \lambda_\alpha = 0;
\ \
q(t) > - \lambda_\alpha \ \ \text{if } \lambda_\alpha > 0;
\ \ \text{and}\ \
q(t) > \frac{\lambda_\alpha}{\me^{- \lambda_\alpha} - 1}
\ \ \text{if } \lambda_\alpha < 0.
\end{equation*}
This gives concrete examples of solvable LQ problems and Riccati equations with indefinite both state and control weights.
Analogous arguments can be moved parallel to the multi-dimensional case. 
\end{example}

\begin{example}\label{ex5-04}
Let us keep on the LQ problem in the previous example, and additionally assume that
\[a = c = q(t) = 0,
\quad \text{and}
\quad b=1,\]
and $r(t) = r$ is a constant.
In this case, we have
\begin{gather*}
\lambda_\alpha = - (1-\alpha)^{-1},
\quad \text{and}
\quad
{\vf}_{\alpha}(t) = \alpha \me^{\frac{t-1}{1-\alpha}}>0.
\end{gather*}
Thus, every $r$ that 
\[r \ge r_0 = - \sup_{\alpha\in(0,1)} \alpha\me^{-\frac{1}{1-\alpha}}
= \alpha\me^{-\frac{1}{1-\alpha}} \big\vert_{\alpha = \frac{3-\sqrt{5}}{2}}
\approx - 0.076\]
ensures the solvability of the problem.
However, such a low bound is larger than that determined in~\cite[Example 3.2]{chen1998stochastic} (there $r_0 \approx -0.1586$) with respect to the same problem.
This indicates in a sense that the characterization given by Theorem~\ref{thm1.1} is still rough although it provides various examples of solvable Riccati equations.
\end{example}

\bibliographystyle{amsalpha}
\bibliography{RicEq}

\providecommand{\bysame}{\leavevmode\hbox to3em{\hrulefill}\thinspace}
\providecommand{\MR}{\relax\ifhmode\unskip\space\fi MR }
\providecommand{\MRhref}[2]{%
  \href{http://www.ams.org/mathscinet-getitem?mr=#1}{#2}
}
\providecommand{\href}[2]{#2}
\begin{thebibliography}{Won68}

\bibitem[Bis76]{bismut1976linear}
J.~M. Bismut, \emph{{Linear quadratic optimal stochastic control with random
  coefficients}}, SIAM Journal on Control and Optimization \textbf{14} (1976),
  no.~3, 419--444.

\bibitem[CLZ98]{chen1998stochastic}
S.~Chen, X.~Li, and X.~Y. Zhou, \emph{{Stochastic linear quadratic regulators
  with indefinite control weight costs}}, SIAM Journal on Control and
  Optimization \textbf{36} (1998), no.~5, 1685--1702.

\bibitem[CZ00]{chen2000stochastic}
S.~Chen and X.~Y. Zhou, \emph{{Stochastic linear quadratic regulators with
  indefinite control weight costs. II}}, SIAM Journal on Control and
  Optimization \textbf{39} (2000), no.~4, 1065--1081.

\bibitem[HZ03]{hu2003indefinite}
Y.~Hu and X.~Y. Zhou, \emph{{Indefinite stochastic Riccati equations}}, SIAM
  Journal on Control and Optimization \textbf{42} (2003), no.~1, 123--137.

\bibitem[Pen92]{peng1992stochastic}
Shige Peng, \emph{{Stochastic Hamilton-Jacobi-Bellman equations}}, SIAM Journal
  on Control and Optimization \textbf{30} (1992), no.~2, 284--304.

\bibitem[QZ13]{qian2013existence}
Z.~Qian and X.~Y. Zhou, \emph{{Existence of solutions to a class of indefinite
  stochastic Riccati equations}}, SIAM Journal on Control and Optimization
  \textbf{51} (2013), no.~1, 221--229.

\bibitem[Tan03]{tang2003general}
S.~Tang, \emph{{General linear quadratic optimal stochastic control problems
  with random coefficients: linear stochastic Hamilton systems and backward
  stochastic Riccati equations}}, SIAM Journal on Control and Optimization
  \textbf{42} (2003), no.~1, 53--75.

\bibitem[Won68]{wonham1968matrix}
W.~M. Wonham, \emph{{On a matrix Riccati equation of stochastic control}}, SIAM
  Journal on Control \textbf{6} (1968), no.~4, 681--697.

\bibitem[YZ99]{young1999stochastic}
J.~Yong and X.~Y. Zhou, \emph{{Stochastic Controls: Hamiltonian Systems and HJB
  Equations}}, vol.~43, Springer, 1999.

\end{thebibliography}

\end{document}